\documentclass[15pt]{article}

\usepackage{amsmath,amsfonts,amssymb,amsthm,amscd,mathrsfs}
\usepackage{enumitem}

\usepackage{yhmath}
\usepackage{graphicx} 

\newcounter{stepctr}
{\end{list}}

\newtheorem{thm}{Theorem}[section]
\newtheorem{prop}[thm]{Proposition}
\newtheorem{cor}[thm]{Corollary}

\theoremstyle{definition}
\newtheorem{dfn}[thm]{Definition}

\newtheorem{rema}[thm]{Remark}
\newtheorem{lem}[thm]{Lemma}
\newtheorem{prob*}{Open problem}
\newcommand{\demo}{\begin{proof}}

\newcommand{\A}{\mathcal{A}}
\newcommand{\NN}{\mathcal{N}}
\newcommand{\R}{\ensuremath{\mathcal{R}}}

\newcommand{\N}{\mathbb{N}}

\usepackage{array,multirow,makecell}
\setcellgapes{1pt}
\makegapedcells
\newcolumntype{R}[1]{>{\raggedleft\arraybackslash }b{#1}}
\newcolumntype{L}[1]{>{\raggedright\arraybackslash }b{#1}}
\newcolumntype{C}[1]{>{\centering\arraybackslash }b{#1}}

\def\ll^2{{\mathcal L}(\ell^2(\N))}

\def\f^0x{{\mathcal F^0}(X) }

\title {\bf  Perturbations   not necessarily commutative}

\author{Zakariae Aznay, Abdelmalek Ouahab, Hassan Zariouh }

\date{}
\begin{document}
\maketitle \thispagestyle{empty}
\begin{abstract}\noindent\baselineskip=10pt
This paper treatises the preservation of some spectra under perturbations not necessarily commutative and generalizes several    results which have been proved in the case of commuting operators.
\end{abstract}

 \baselineskip=15pt
  \baselineskip=15pt
\section{Introduction and preliminaries}
According to \cite{aznay-ouahab-zariouh7}, for an element $a$    of a ring  $\A,$ denote by 
    $$\mbox{comm}_{l}(a)=\{b \in \A :  ab \in \mbox{comm}(a) \text{ and } ba  \in \mbox{comm}(b)\},$$
  $$ \mbox{ by }\mbox{comm}_{r}(a) =\{b \in \A :    ab \in \mbox{comm}(b) \text{ and  } ba \in \mbox{comm}(a)\},$$
  $$\mbox{  and  by } \mbox{comm}_{w}(a)=\mbox{comm}_{l}(a)\cap\mbox{comm}_{r}(a),$$
where $\mbox{comm}(a)$ is    the set of all elements that commute   with $a.$ Denote also   by  $\mbox{comm}^{2}(a)=\mbox{comm}(\mbox{comm}(a))$ and by $\mbox{Nil}(\A)$ the nilradical  of $\A.$ If  $\A$ is a unital  complex Banach algebra,   we means by $\sigma(a),$ $\mbox{acc}\,\sigma(a),$    $r(a)$ and $\mbox{exp}(a),$ the  spectrum of $a,$ the accumulation points of $\sigma(a),$  the spectral radius of $a$   and the exponential of $a,$  respectively. We say that  $a$ is quasi-nilpotent  if $r(a)=0.$   If $T\in L(X)$ the   algebra of all bounded linear operators acting on an infinite dimensional complex Banach space  $X,$ then     $T^{*},$ $\alpha(T)$ and  $\beta(T)$ means   respectively,     the dual of  $T,$   the dimension of the kernel $\mathcal{N}(T)$   and the codimension of the range $\R(T),$  and  denote by  $\R(T^{\infty})=\underset{n\geq0}{\bigcap}\R(T^{n})$ and $\mathcal{N}(T^{\infty})=\underset{n\geq0}{\bigcup}\mathcal{N}(T^{n}).$ Moreover,  the ascent and the descent of $T$ are defined  by $p(T)=\inf\{n\in \mathbb{N}: \mathcal{N}(T^n) = \mathcal{N}(T^{n+1})\}$    and  $q(T)= \inf\{n\in \mathbb{N}: \R(T^n) = \R(T^{n+1})\}$  (with $\mbox{inf}\emptyset=\infty$).   We say that a  subspace $M$ of $X$ is $T$-invariant if $T(M)\subset M$ and   the restriction of $T$ on $M$ is denoted by $T_{M},$ and we say that   $(M,N) \in \mbox{Red}(T)$  if $M,$   $N$ are  closed $T$-invariant subspaces and  $X=M\oplus N.$     For  $n\in\N,$   denote by   $T_{[n]}=T_{\mathcal{R}(T^{n})}$    and by     $m_T=\mbox{inf}\{n \in \N  :  \mbox{min}\{\alpha(T_{[n]}),\beta(T_{[n]})\}<\infty\}$   the \textit{essential  degree} of   $T.$ An operator   $T$ is called upper semi-B-Fredholm  (resp., lower semi-B-Fredholm)  if the \textit{essential ascent} $p_{e}(T):=\mbox{inf}\{n \in \N  :  \alpha(T_{[n]})<\infty\}<\infty$ and $\R(T^{p_{e}(T)+1})$ is closed  (resp.,  the \textit{essential descent}  $q_{e}(T):=\mbox{inf}\{n \in \N  :  \beta(T_{[n]})<\infty\}<\infty$ and $\R(T^{q_{e}(T)})$ is closed). If  $T$ is   an upper or a lower  (resp.,  upper and   lower)  semi-B-Fredholm, then $T$ it is called   \emph{semi-B-Fredholm} (resp., \emph{B-Fredholm})  and      its     index   is  defined   by $\mbox{ind}(T) = \alpha(T_{[m_{T}]})-\beta(T_{[m_{T}]}).$ $T$ is said to be an upper semi-B-Weyl (resp.,  lower semi-B-Weyl, B-Weyl,  left Drazin invertible, right Drazin invertible, Drazin invertible) if  $T$ is an upper semi-B-Fredholm with $\mbox{ind}(T)\leq 0$  (resp., $T$ is a lower  semi-B-Fredholm with $\mbox{ind}(T)\geq 0,$ $T$ is a B-Fredholm with $\mbox{ind}(T)=0,$   $T$ is an upper semi-B-Fredholm and $p(T_{[m_{T}]})<\infty,$ $T$ is a lower  semi-B-Fredholm and $q(T_{[m_{T}]})<\infty,$  $p(T_{[m_{T}]})=q(T_{[m_{T}]})<\infty$).   If $T$ is  upper semi-B-Fredholm (resp.,  lower semi-B-Fredholm,    semi-B-Fredholm, B-Fredholm, upper semi-B-Weyl,  lower semi-B-Weyl, B-Weyl,  left Drazin invertible, right Drazin invertible, Drazin invertible) with essential degree $m_{T}=0,$ then $T$ is said to be  an upper semi-Fredholm (resp.,  lower semi-Fredholm,   semi-Fredholm, Fredholm, upper semi-Weyl, lower semi-Weyl, Weyl, upper semi-Browder, lower semi-Browder, Browder).  $T$ is said to be  bounded below if $T$ is upper semi-Fredholm with $\alpha(T)=0,$ and is said  to be   Riesz   if $T-\lambda I$ is  Fredholm    for all non-zero complex  $\lambda$ or  equivalently  $\pi(T):=T+K(X)$ is quasi-nilpotent in the Calkin algebra $L(X)/K(X)$ ($K(X)$ is the  ideal  of  compact operators). Following \cite{rwassariesz}, $T$ is said  to be  generalized Drazin-Riesz invertible if there exists $(M,N)\in \mbox{Red}(T)$ such that $T_{M}$ is invertible and $T_{N}$ is Riesz.  $T$ is said  to be   semi-regular (resp., essentially semi-regular) if $\R(T)$ is closed and $\NN(T)\subseteq \R(T^{\infty})$ (resp., $\R(T)$ is closed and there
exists a finite-dimensional subspace $F$  such that  $\NN(T)\subseteq \R(T^{\infty})+F$).

 \par In  this paper, we study the stability of the   spectra summarized in the next list     under the     algebraic conditions considered in \cite{aznay-ouahab-zariouh7} that are  weaker than the commutativity.

\medskip
 \begin{tabular}{l|l}
   $\sigma_{a}(T)$: approximative spectrum of $T$ &  $\sigma_{bw}(T)$: B-Weyl spectrum of $T$\\
   $\sigma_{e}(T)$:  essential  spectrum of $T$ &  $\sigma_{ubw}(T)$: upper semi-B-Weyl spectrum of $T$\\
   $\sigma_{uf}(T)$:  upper semi-Fredholm spectrum of $T$ &  $\sigma_{lbw}(T)$: lower semi-B-Weyl spectrum of $T$\\
   $\sigma_{lf}(T)$:  lower semi-Fredholm spectrum of $T$ &  $\sigma_{ld}(T)$: left Drazin spectrum of $T$\\
   $\sigma_{w}(T)$: Weyl spectrum of $T$ & $\sigma_{rd}(T)$:  right Drazin spectrum of $T$\\
   $\sigma_{uw}(T)$: upper semi-Weyl spectrum of $T$ & $\sigma_{d}(T)$:  Drazin spectrum of $T$\\
   $\sigma_{lw}(T)$: lower semi-Weyl spectrum of $T$ & $\sigma_{se}(T)$:  semi-regular spectrum of $T$\\
   $\sigma_{b}(T)$:  Browder spectrum of $T$ &  $\sigma_{gd}(T)$:  generalized Drazin spectrum of  $T$\\
   $\sigma_{ub}(T)$:  upper semi-Browder spectrum of $T$ &  $\sigma_{g_{z}d}(T)$:   $g_{z}$-invertible   spectrum of $T$ \cite{aznay-ouahab-zariouh6}\\
   $\sigma_{lb}(T)$:  lower semi-Browder spectrum of $T$ &   $\sigma_{bf}(T)$: B-Fredholm spectrum of $T$\\
 $\sigma_{ubf}(T)$: upper semi-B-Fredholm spectrum of $T$ &   $\sigma_{lbf}(T)$: lower  semi-B-Fredholm spectrum of $T$\\
\end{tabular}\\

\noindent  As an extension of  \cite[Proposition 2.6]{berkani},  we prove that if    $a,b\in \A$ are Drazin invertible   such that $a \in \mbox{comm}_{w}(b),$ then $a^{D}\in\mbox{comm}(b^{D})$ and $ab$ is Drazin invertible with  $(ab)^{D}=a^{D}b^{D}.$ Moreover, we prove that if  $T$ is generalized Drazin-Riesz invertible and $R$ is a Riesz operator such that   $R\in\mbox{comm}_{l}(T)$ and $T \in \mbox{comm}(TR)$ (or $R\in\mbox{comm}_{r}(T)$ and $T \in \mbox{comm}(RT)$), then  $T+R$ is generalized Drazin-Riesz invertible and   $\sigma_{*}(T)=\sigma_{*}(T+R),$ where  $\sigma_{*}\in\{\sigma_{e},\sigma_{uf},\sigma_{lf},\sigma_{w},\sigma_{uw},\sigma_{lw},\sigma_{b},\sigma_{ub},\sigma_{lb}\}.$ If in addition $T$ is generalized Drazin invertible and $R$ is  quasi-nilpotent,  then   $T+R$ is generalized Drazin invertible and  $\sigma(T)=\sigma(T+R).$
  Among other results,  we give  a new characterization of power finite rank operators,   by   proving that  if $\sigma_{*}\in\{\sigma_{bf},\sigma_{ubf},\sigma_{lbf},\sigma_{d},\sigma_{ld},\sigma_{rd}\},$ then  $F$ is  a power finite rank operator  if and only if   $\sigma_{*}(T)=\sigma_{*}(T+F)$ for every    generalized Drazin-Riesz invertible operator $T\in\mbox{comm}_{w}(F).$

\section{Pseudo invertible elements of  a ring}
\begin{dfn} Let $\A$ be a ring and let $a\in \A.$ We say that $a$ is pseudo invertible if there exists  $c \in \mbox{comm}^{2}(a)$  such that    $c=c^{2}a.$ In this case we say that $c$ is   a  pseudo inverse of $a.$
\end{dfn}
\noindent According to \cite{koliha}, an element  $a$ of a unital   complex  Banach algebra   $\A,$ is said to be    generalized Drazin invertible if there exists  $b \in  \mbox{comm}\,(a)$ such that       $b^{2}a=b$ and $a-a^{2}b$ is quasi-nilpotent. If so then  $b \in  \mbox{comm}^{2}\,(a)$  and is  denoted by $a^{D}$ and  called the generalized Drazin inverse of $a.$ It is proved also that   $a$ is generalized Drazin invertible if and only if $0\notin\mbox{acc}\,\sigma(a).$ So every generalized Drazin invertible element $a$ is pseudo invertible and its Drazin  inverse $a^{D}$  is  a pseudo inverse of $a.$
\begin{prop}\label{proppseudoring}  Let $\A$ be a ring and let  $a,b \in \A$ such that $a$ is pseudo invertible. Then for every pseudo inverse           $c$ of $a,$  we have\\
(i) If  $ba \in \mbox{comm}(a),$ then $bc \in \mbox{comm}(c),$ and  if in addition      $a\in\mbox{comm}_{r}(b),$  then $c\in\mbox{comm}_{r}(b).$\\
(ii) If  $ab \in \mbox{comm}(a),$ then $cb \in \mbox{comm}(c),$ and   if in addition   $a\in\mbox{comm}_{l}(b),$  then $c\in\mbox{comm}_{l}(b).$\\
(iii) If   $a\in\mbox{comm}_{w}(b),$  then $c\in\mbox{comm}_{w}(b).$
\end{prop}
\begin{proof} By hypotheses we get    $c^{2}=c^{2}ca=acc^{2}.$   Thus  \\
(i) If  $ba \in \mbox{comm}(a),$ then   $ba \in \mbox{comm}(c)$ and  so $bc^{2}=((ba)c)c^{2}=(c(ba))c^{2}=cbc.$ If in addition   $a\in\mbox{comm}_{r}(b),$  then $bcb=((ba)c^{2})b=c^{2}(bab)=cb^{2}.$ So  $c\in\mbox{comm}_{r}(b).$\\
(ii) If   $ab \in \mbox{comm}(a),$ then   $ab \in \mbox{comm}(c)$ and thus  $c^{2}b=c^{2}(c(ab))=c^{2}((ab)c)=cbc.$  If in addition   $a\in\mbox{comm}_{l}(b),$  then $bcb=b(c^{2}(ab))=(bab)c^{2}=b^{2}ac^{2}=b^{2}c.$ So $c\in\mbox{comm}_{l}(b).$\\
(iii) Follows directly from  the previous points.
\end{proof}
 From Proposition \ref{proppseudoring}, we immedialtely deduce  the next corollary.
\begin{cor}  Let $\A$ be a unital  complex Banach algebra     and let  $a,b \in \A$  such that $a$ is generalized Drazin invertible. Then  the  following assertions hold:
\begin{enumerate}[nolistsep]
\item[(i)] If  $ba \in \mbox{comm}(a),$ then $ba^{D} \in \mbox{comm}(a^{D}),$  and if in addition      $a\in\mbox{comm}_{r}(b),$  then $a^{D}\in\mbox{comm}_{r}(b).$
\item[(ii)] If  $ab \in \mbox{comm}(a),$ then $a^{D}b \in \mbox{comm}(a^{D}),$ and  if in addition   $a\in\mbox{comm}_{l}(b),$  then $a^{D}\in\mbox{comm}_{l}(b).$
\item[(iii)] If   $a\in\mbox{comm}_{w}(b),$  then $a^{D}\in\mbox{comm}_{w}(b).$
\end{enumerate}
\end{cor}
\begin{lem}\label{propaboutinverses}
Let $\A$ be a ring and let  $a,b \in \A$  pseudo invertible. If  $c$  and $d$ are respectively   pseudo inverses of $a$ and $b,$ we  then have  \\
(i) If       $a\in\mbox{comm}_{r}(b),$  then $c\in\mbox{comm}_{r}(b)$ and  $cb,ad,cd\in \mbox{comm}(b).$\\
(ii) If    $a\in\mbox{comm}_{l}(b),$  then $c\in\mbox{comm}_{l}(b)$ and $cb,ad,cd\in \mbox{comm}(a).$\\
(iii) If   $a\in\mbox{comm}_{w}(b),$  then $c\in\mbox{comm}_{w}(b)$  and  $cb,ad,cd\in \mbox{comm}(a)\cap\mbox{comm}(b).$
\end{lem}
\begin{proof} (i) Assume that $a\in\mbox{comm}_{r}(b).$ Since $ba\in \mbox{comm}(a)$ then $cb^{2}=c^{2}ab^{2}=c^{2}bab=(c^{2}(ba))b=bcb.$ Thus $cb\in \mbox{comm}(b).$ On the other hand, we have  $adb=abd=ab^{2}d^{2}=babd^{2}=bad$ and  $bcd=((ba)c^{2})d=c^{2}(b(ad))=c^{2}adb=cdb.$     Then $ad,cd\in \mbox{comm}(b).$ The point (ii) goes similarly with the first point  and the third point is clear.
\end{proof}
Let $\A$ be a ring. An element  $a\in \A$ is said to be Drazin  invertible of degree $n$  if  there exists $c \in \mbox{comm}(a)$ such that $c=c^{2}a$ and  $a^{n+1}c=a^{n}.$ In this case $c=a^{D},$  $c \in  \mbox{comm}^{2}\,(a)$ and if    $a$ is of degree $n$ and is not of degree $n-1,$ then $n$  is called the index of $a$ and  is denoted by $i(a)=n.$ For more details about this definition, we refer the reader to  \cite{Aiena1}.
Our  next proposition gives an extension  of  \cite[Proposition 2.6]{berkani}.
\begin{prop}\label{productDrazin}  Let  $a,b$ two Drazin invertible elements of   a ring $\A.$  The following assertions hold:\\
(i) If       $a\in\mbox{comm}_{r}(b)$ and $ab\in \mbox{comm}(a),$  then $ab$ is Drazin invertible and $(ab)^{D}=b^{D}a^{D}.$\\
(ii) If    $a\in\mbox{comm}_{l}(b),$  then $ab$ is Drazin invertible and $(ab)^{D}=a^{D}b^{D}.$\\
(iii) If   $a\in\mbox{comm}_{w}(b),$  then $ab$ is Drazin invertible, $(ab)^{D}=a^{D}b^{D}$ and $a^{D}\in\mbox{comm}(b^{D}).$
\end{prop}
\begin{proof}(i)  Assume that  $a\in\mbox{comm}_{r}(b)$ and let $n=\mbox{max}\{i(a),i(b)\}.$ From \cite[Lemma 3.1]{aznay-ouahab-zariouh7}  we deduce that  $b^{D}a^{D}(ab)^{n+1}=b^{D}a^{D}a^{n+1}b^{n+1}=b^{D}a^{n}b^{n+1}=(b^{D}(a^{n}b))b^{n}=a^{n}bb^{D}b^{n}=a^{n}b^{n}=(ab)^{n}.$ On the other hand, since  $ab\in \mbox{comm}(a)$ then $((ab)b^{D})a^{D})=b^{D}((ab)a^{D})=(b^{D}a^{D})ab.$ It follows from  Lemma \ref{propaboutinverses} that  $b^{D}a^{D}\in \mbox{comm}(a).$  Hence      $(b^{D}a^{D})(ab)(b^{D}a^{D})=(a(b^{D}a^{D}))bb^{D}a^{D}=ab^{D}((a^{D}b^{D})b)a^{D}=ab^{D}b(a^{D}(b^{D}a^{D}))=ab^{D}bb^{D}(a^{D})^{2}=ab^{D}(a^{D})^{2}=b^{D}a^{D}.$ This implies that  $ab$ is Drazin invertible, $(ab)^{D}=b^{D}a^{D}$ and $i(ab)\leq \mbox{max}\{i(a),i(b)\}.$  Also note that  (again from  Lemma \ref{propaboutinverses}) that  $(b^{D}a^{D})ab=b(b^{D})^{2}a^{D}ab=bb^{D}((b^{D}a^{D})a)b=b(b^{D}(ab^{D}))a^{D}b=ba((b^{D})^{2}(a^{D}b))=baa^{D}b(b^{D})^{2}=ba(a^{D}b^{D}).$ The point (ii) is done in    \cite[Theorem 3.1]{ZhuDrazin} and  the point  (iii) is clear.
\end{proof}
\begin{prop} Let $\A$ be a ring and let $a\in \A.$   If  there exists  $b \in \A$ and $n\in \N$ such that $a\in\mbox{comm}(ab)\cap\mbox{comm}(ba)$, $bab=b$ and $ba^{n+1}=a^{n},$ then $a$ is Drazin invertible and $a^{D}=b.$
\end{prop}
\begin{proof} As $a\in\mbox{comm}(ab)\cap\mbox{comm}(ba)$  then by \cite[Lemma 3.1, Remark 3.3]{aznay-ouahab-zariouh7}, we deduce that     $ba=(ba)^{n+2}=b^{n+2}a^{n+2}=a^{n+2}b^{n+2}=(ab)^{n+2}=ab.$   Thus  $b\in \mbox{comm}(a).$
\end{proof}

\section{ Perturbations of pseudo invertible  operators}

According to \cite{mbekhta}, the analytic core $\mathcal{K}(T)$  and the quasi-nilpotent part $\mathcal{H}_{0}(T)$  of $T\in L(X)$ are defined by
$$\mathcal{K}(T)=\{x \in X : \exists \epsilon>0, \exists (x_{n})_{n} \subset X \text{ such that } \forall n \in \N \,\, x=x_{0}, Tx_{n+1}=x_{n} \text{ and } \underset{n}{\sup}\, \|x_{n}\|^{\frac{1}{n}}<\infty \}$$
$$ \mbox{ and }\mathcal{H}_{0}(T)=\{x \in X : \lim\limits_{n \rightarrow \infty}\,\|T^{n}x\|^{\frac{1}{n}}=0\}.$$
\begin{lem}\label{lemmmf.1} Let $S, T \in L(X).$ The following assertions hold:\\
(i) If     $ST\in\mbox{comm}(T),$ then  $\R(T^{\infty})$  and $\mathcal{K}(T)$ are   $S$-invariant.\\
(ii) If     $TS\in\mbox{comm}(T),$ then $\mathcal{N}(T^{\infty})$ and   $\mathcal{H}_{0}(T)$    are $S$-invariant.
\end{lem}
\begin{proof}(i)  Let $n\geq 1$ and let  $x\in S(\R(T^{n+1})).$ Then  there exists $z$ such that $x=ST^{n+1}z,$  and since $TST=ST^{2}$ then $x=T^{n}STz\in \R(T^{n}).$  So   $S(\R(T^{n+1}))\subset \R(T^{n}).$ Hence   $\R(T^{\infty})$  is  $S$-invariant. If $x\in \mathcal{K}(T),$ then there exists $(x_{n})\subset X$ such that $x=x_{0},$  $x_{n}=Tx_{n+1}$ for every  $n\in \N$ and  $\underset{n}{\sup}\, \|x_{n}\|^{\frac{1}{n}}<\infty.$  We put    $y_{n}=Sx_{n}$ for every  $n\in \N.$ We then have   $y_{n}=Sx_{n}=TSTx_{n+2}=Ty_{n+1}$ and  $\underset{n}{\sup}\, \|y_{n}\|^{\frac{1}{n}}\leq \mbox{max}\,\{\|S\|,1\} \underset{n}{\sup}\, \|x_{n}\|^{\frac{1}{n}}<\infty.$  Therefore $Sx  \in \mathcal{K}(T)$ and  then  $\mathcal{K}(T)$ is  $S$-invariant.\\
 (ii) Since  $S(\mathcal{N}(T^{n}))\subset \mathcal{N}(T^{n+1})$ for every  $n>0$ then $\mathcal{N}(T^{\infty})$ is $S$-invariant.   Let $x\in \mathcal{H}_{0}(T),$ then  $\|T^{n+1}(Sx)\|^{\frac{1}{n+1}}=\|TST^{n}x\|^{\frac{1}{n+1}}\leq \|TS\|^{\frac{1}{n+1}}\|T^{n}x\|^{\frac{1}{n+1}}=\|TS\|^{\frac{1}{n+1}}(\|T^{n}x\|^{\frac{1}{n}})^{\frac{n}{n+1}}$ for every integer $n>0.$ Thus $Sx\in \mathcal{H}_{0}(T)$ and  then  $\mathcal{H}_{0}(T)$   is $S$-invariant.
\end{proof}
The next corollary is a  consequence of the previous lemma  and  Proposition \ref{proppseudoring}.
\begin{cor}\label{corinvariantinverse} Let $T \in L(X)$ be a pseudo invertible operator and if  $L$ is   a pseudo inverse of $T,$ then    the following assertions hold:\\
(i) If     $ST\in\mbox{comm}(T),$ then  $\R(L^{\infty})$ and $\mathcal{K}(L)$ are   $S$-invariant.\\
(ii) If     $TS\in\mbox{comm}(T),$ then $\mathcal{N}(L^{\infty})$ and   $\mathcal{H}_{0}(L)$    are $S$-invariant.
\end{cor}
\begin{thm}\label{thm.drn}  If  $T\in L(X)$  and  $N\in \mbox{Nil}(L(X))$ such that    $N\in\mbox{comm}_{l}(T)$ and $T\in \mbox{comm}(NT)$ (or $N\in\mbox{comm}_{r}(T)$ and  $T\in \mbox{comm}(TN)$),   then  $T$ is Drazin invertible   if and only $T+N$ is Drazin invertible.  In  this  case,  we have  $\mathcal{N}(T^{\infty})=\mathcal{N}((T+N)^{\infty})$ and $\R(T^{\infty})=\R((T+N)^{\infty}).$
\end{thm}
\begin{proof} Suppose that  $T$  is Drazin invertible, that is,  $p:=p(T)=q(T)<\infty.$ Then  $(A,B):=(\R(T^{p}),\mathcal{N}(T^{p}))\in Red(T),$ $T_{A}$ is invertible and $T_{B}$ is nilpotent.  From  Lemma \ref{lemmmf.1}  we deduce that  $(A,B)\in Red(N),$      and so   $T=T_{A}\oplus T_{B}$ and $N=N_{A}\oplus N_{B}.$  Therefore  $T+N=(T+N)_{A}\oplus  (T+ N)_{B}.$ By hypotheses and the fact that   $T_{A}$ is invertible and $N_{A}$ is nilpotent, we conclude that   $N_{A}\in\mbox{comm}(T_{A}).$  Thus    $(T+N)_{A}$ is  invertible and by \cite[Lemma  3.8]{aznay-ouahab-zariouh7}, it follows that       $T+N$ is Drazin invertible.  The converse is clear.
\end{proof}

For $T\in L(X),$  denote   by $\mbox{IRed}(T)=\{(M,N)\in \mbox{Red}(T): T_{M} \text{ is invertible}  \text{ and }  (M,N)\in \mbox{Red}(U) \text{ for all }   U\in \mbox{comm}(T) \}.$ If   $T$ is   pseudo invertible,  denote by    $\mbox{PI}(T)$   the set of  its  pseudo inverses.    Note that the  class  of pseudo invertible operators  is much broader, it contains in particular the class of   $g_{z}$-invertible operators, see \cite[Remark 4.19]{aznay-ouahab-zariouh6}.
\begin{prop}\label{34} $T \in L(X)$ is pseudo invertible if and only if there exists $(M,N)\in \mbox{Red}(U)$ such that $T_{M}$ is invertible for every  $U\in \mbox{comm}(T).$ If this is the case,   the map  $\Phi: \mbox{IRed}(T) \longrightarrow \mbox{PI}(T)$ defined by $\Phi(M,N)=(T_{M})^{-1}\oplus 0_{N}$  is  onto.
\end{prop}
\begin{proof} Assume that $T$ is pseudo invertible and let   $S \in \mbox{PI}(T).$   Then $TS$ is a projection and    $(M,N):=(\R(TS),\mathcal{N}(TS)) \in \mbox{Red}(T).$ Let $U\in \mbox{comm}(T),$ then $U(M)=\R(UTS)=\R(TSU)\subset M$ and   $U(N)\subset N.$  Thus  $(M,N)\in \mbox{Red}(U).$  Moreover, if $x\in \mathcal{N}(T_{M}),$ then $x=TSy$ and $Tx=0.$ Therefore $x=(TS)^2y=STx=0.$  If $x=TSy\in M,$ then   $x=T(TS(Sy))\in T(M).$ Thus       $T_{M}$ is  invertible. Let us  show that $S=(T_{M})^{-1}\oplus 0_{N}.$ We have   $S_{N}=0_{N},$ since   $S=STS.$    Let $x=TSy \in M.$  As $Sy=STSy \in M$ then $Sx=Sy=(T_{M})^{-1}T_{M}Sy=(T_{M})^{-1}x.$ Hence $S=(T_{M})^{-1}\oplus 0_{N}.$  Conversely,   if   $(M,N)\in \mbox{IRed}(T),$  then  the operator  $S=(T_{M})^{-1}\oplus 0_{N}$ gives  the desired result. Indeed,   it is clear that    $S=S^{2}T,$ and  if  $A\in \mbox{comm}(T),$ then $(M,N)\in \mbox{Red}(A).$  So  $A_{M}\in \mbox{comm}((T_{M})^{-1}).$ Therefore $A \in \mbox{comm}{(S)}$ and consequently  $S \in \mbox{comm}^{2}(T).$
\end{proof}
It follows from  Proposition \ref{34} that every pseudo inverse $S$ of a  pseudo invertible operator  $T$ is Drazin invertible with $p(S)=q(S)\leq 1.$  Hence $\mbox{IRed}(T)\subset\{(\R(S),\mathcal{N}(S)): S \in \mbox{PI}(T)\}.$  Moreover, if $T$ is generalized Drazin invertible, then  $(\mathcal{K}(T),\mathcal{H}_{0}(T))\in \mbox{IRed}(T).$
\begin{cor}\label{corired} Let $T \in L(X).$ Then   $\mbox{IRed}(T)= \{(M,N)\in \mbox{Red}(T): T_{M} \text{ is invertible}  \text{ and }  (M,N)\in \mbox{Red}(U) \text{ for all }   U\in  L(X) \text{ such that }  T\in  \mbox{comm}(TU)\cap\mbox{comm}(UT)\}.$
\end{cor}
\begin{proof} Let $C_{T}=\{(M,N)\in \mbox{Red}(T): T_{M} \text{ is invertible}  \text{ and }  (M,N)\in \mbox{Red}(U) \text{ for all }   U\in  L(X)  \text{ such } \newline \text{that }    T\in  \mbox{comm}(TU)\cap\mbox{comm}(UT)\}.$ It is clear that $C_{T}\subset\mbox{IRed}(T).$  For the opposite inclusion,  assume that $T$ is pseudo invertible (in the  other case   $\mbox{IRed}(T)=\emptyset$). Let  $(M,N)\in \mbox{IRed}(T)$ and let  $U\in L(X)$ such that $T\in \mbox{comm}(TU)\cap\mbox{comm}(UT)$  and we  consider  the pseudo inverse  operator  $S=\Phi(M,N)$   of $T$ associated to $(M,N).$    It is clear that  $S$ is Drazin invertible, and  hence $(M,N)=(\R(S),\mathcal{N}(S))=(\R(S^{\infty}),\mathcal{N}(S^{\infty})).$ As  $T \in \mbox{comm}(TU)\cap\mbox{comm}(UT),$   $S \in \mbox{comm}^{2}(T)$  and  $S=S^{2}T,$  it follows from    Corollary \ref{corinvariantinverse}  that $(M,N) \in \mbox{Red}(U).$
\end{proof}
\begin{prop}\label{propriesz}
Let  $S,T \in L(X).$ The following assertions hold:\\
(i)  If $T$ or $S$ is  Riesz and  $TS \in \mbox{comm}(T)\cup\mbox{comm}(S),$   then  $TS$  is  Riesz.\\
(ii)   If $T$ and $S$ are Riesz and  $S \in \mbox{comm}_{r}(T)\cup  \mbox{comm}_{l}(T),$ then    $T + S$ is Riesz.
\end{prop}
\begin{proof} As mentioned above,  $R\in L(X)$ is Riesz if and only if $\pi(R)$ is quasi-nilpotent in the Calkin algebra $L(X)/K(X).$    The proof  is then  a  consequence of  \cite[Corollary  3.10]{aznay-ouahab-zariouh7}.
\end{proof}
   Our next theorem generalizes some  known commutative  perturbation results.

\begin{thm}\label{propgenriesznilp} Let $R, T\in L(X)$  such that  $R$ is Riesz,  $R\in\mbox{comm}_{l}(T)$ and $T \in \mbox{comm}(TR)$ {\rm[}or $R\in\mbox{comm}_{r}(T)$ and $T \in \mbox{comm}(RT)${\rm]}. Then  $T$ is generalized Drazin-Riesz invertible if and only if $T+R$ is generalized Drazin-Riesz invertible.  If this is the case, that is, $T$ is generalized Drazin-Riesz invertible,  then  \\
(i) $\sigma_{*}(T)=\sigma_{*}(T+R),$ where  $\sigma_{*}\in\{\sigma_{e},\sigma_{uf},\sigma_{lf},\sigma_{w},\sigma_{uw},\sigma_{lw},\sigma_{b},\sigma_{ub},\sigma_{lb}\}.$\\
(ii) If $R\in \mbox{Nil}(L(X)),$ then    $\sigma_{+}(T)\setminus\{0\}=\sigma_{+}(T+R)\setminus\{0\}$ and $\sigma_{++}(T)=\sigma_{++}(T+R),$  where  $\sigma_{+}\in\{\sigma_{bf},\sigma_{ubf},\sigma_{lbf},\sigma_{bw},\sigma_{ubw},\sigma_{lbw},\sigma_{d},\sigma_{ld},\sigma_{rd}\}$ and  $\sigma_{++}\in\{\sigma_{bf},\sigma_{bw},\sigma_{d}\}.$  If in addition $X$ is a Hilbert space, then  $\sigma_{+}(T)=\sigma_{+}(T+R).$\\
(iii) If $R$ is in particular  quasi-nilpotent, then $\mbox{acc}\,\sigma_{-}(T)\setminus\{0\}=\mbox{acc}\,\sigma_{-}(T+R)\setminus\{0\},$ where  $\sigma_{-}\in\{\sigma,\sigma_{a},\sigma_{s}\}.$ If  in addition $T$ is generalized Drazin invertible, then $T+R$ is generalized Drazin invertible,  $\sigma(T)=\sigma(T+R),$ $\sigma_{\times}(T)\setminus\{0\}=\sigma_{\times}(T+R)\setminus\{0\}$  and  $\mbox{acc}\,\sigma_{\times}(T)=\mbox{acc}\,\sigma_{\times}(T+R),$ where  $\sigma_{\times}\in\{\sigma_{a},\sigma_{s}\}.$
\end{thm}
\begin{proof} (i) If $T$ is generalized Drazin-Riesz invertible, then there exists $(M,N)\in  \mbox{IRed}(T)$ such that  $T_{N}$ is Riesz. Suppose that $R\in\mbox{comm}_{l}(T)$ and $T \in \mbox{comm}(TR)$ [the case $R\in\mbox{comm}_{r}(T)$ and $T \in \mbox{comm}(RT)$ goes similarly]. From  Corollary \ref{corired} we have   $(M,N)\in \mbox{Red}(R).$  Thus  $T_{M}R_{M}=R_{M}T_{M},$    $R_{N}\in\mbox{comm}_{l}(T_{N})$ and $T_{N} \in \mbox{comm}(T_{N}R_{N}).$   Hence $(T+R)_{M}$ is Browder, and by  Proposition \ref{propriesz},  $(T+R)_{N}$ is Riesz.  From \cite[Theorem 2.3]{rwassariesz}, we deduce  that $T+R$ is generalized Drazin-Riesz invertible.  On the other hand, as $\sigma_{*}(T_{N})\subset\sigma_{b}(T_{N})\subset\{0\}$ then  $\sigma_{*}(T)\setminus\{0\}=(\sigma_{*}(T_{M})\cup\sigma_{*}(T_{N}))\setminus\{0\}=\sigma_{*}(T_{M})\setminus\{0\}=\sigma_{*}((T+R)_{M})\setminus\{0\}=\sigma_{*}(T+R)\setminus\{0\}.$    If $0\notin \sigma_{*}(T),$ then $T$ is semi-Fredholm. From  \cite[Corollary 3.7]{aznay-ouahab-zariouh5},  there exists $(M^{'},N^{'})\in\mbox{Red}(T)$ such that $T_{M^{'}}$ is semi-regular, $T_{N^{'}}$ is nilpotent and $\mbox{dim}\,N^{'}<\infty.$  This entails by \cite[Proposition 2.10]{aznay-ouahab-zariouh6} that $T_{M^{'}}$ is invertible. Hence $T$ is Browder,    $(M^{'},N^{'})=(\R(T^{\infty}),\NN(T^{\infty}))$  and  $0\notin \sigma_{*}(T+R).$   Thus  $\sigma_{*}(T)=\sigma_{*}(T+R).$
\\
(ii) Assume  that $R\in \mbox{Nil}(L(X)).$ We have  $\sigma_{+}(T)\setminus\{0\}=(\sigma_{+}(T_{M})\cup\sigma_{+}(T_{N}))\setminus\{0\}=\sigma_{+}(T_{M})\setminus\{0\}=\sigma_{+}((T+R)_{M})\setminus\{0\}=\sigma_{+}(T+R)\setminus\{0\},$ since $\sigma_{+}(T_{N})\subset\sigma_{b}(T_{N})\subset\{0\}$ and $\sigma_{+}$ is  stable under commuting nilpotent perturbations.  If 0$\notin \sigma_{++}(T),$ then $T$ is B-Fredholm which implies from \cite[Theorem 2.21]{aznay-ouahab-zariouh4} and    \cite[Proposition 2.10]{aznay-ouahab-zariouh6} that $T$ is Drazin invertible. We conclude from Theorem \ref{thm.drn}    that  $\sigma_{++}(T)=\sigma_{++}(T+R).$   If in addition $X$ is a Hilbert space,   using  \cite[Theorem 2.6]{berkani-sarih} and the same argument as above, we deduce that $\sigma_{+}(T)=\sigma_{+}(T+R).$  The point (iii) goes similarly and is left to the reader.
\end{proof}

Let $T\in L(X)$ and let $Q\in L(X)$ be a   quasi-nilpotent operator which commutes with $T.$  It is well known that  $\sigma(T)=\sigma(T+Q).$  The proof of Theorem \ref{propgenriesznilp} suggests the following question. \\
{\bf Question:} This equality
$\sigma(T)=\sigma(T+Q)$ remains true if we only have  $Q\in\mbox{comm}_{l}(T)$ and $T \in \mbox{comm}(TQ)$ [or $Q\in\mbox{comm}_{r}(T)$ and $T \in \mbox{comm}(QT)$]?  Note that  Theorem \ref{propgenriesznilp}  gives a partial  answer to  this  question. \\
\par We give in the next result an  extension of  \cite[Lemma 3.81]{Aiena1}.  Recall that $T\in L(X)$ is said to be  algebraic if there exists a non-null polynomial $P$ such that $P(T)=0.$ 
\begin{cor} If $T\in L(X)$ is algebraic and   $N\in \mbox{Nil}(L(X))$ such that   $N\in\mbox{comm}_{l}(T)$ and $T \in \mbox{comm}(TN)$ {\rm[}or $N\in\mbox{comm}_{r}(T)$ and $T \in \mbox{comm}(NT)${\rm]},  then $T+N$ is algebraic.
\end{cor}
\begin{proof} The proof  follows from   Theorem \ref{propgenriesznilp} and the fact that   $T$ is algebraic if and only if 
it has empty Drazin spectrum.
\end{proof}
\begin{thm}\label{thm.srn} Let $R \in L(X)$ and   $T \in \mbox{comm}(TR)\cap \mbox{comm}(RT).$  The following assertions hold:\\
(i) If    $T$ is essentially semi-regular and  $R$ is  Riesz,     then $T+R$ is  essentially semi-regular.\\
(ii) If    $T$ is semi-regular and  $R=Q$ is  quasi-nilpotent, then  $T+Q$ is semi-regular.
\end{thm}
\begin{proof} Let $M= \R(T^{\infty}).$\\
(i)   Assume that  $T$ is essentially semi-regular.  From  \cite[Proposition 13]{kordula-muller} and Lemma \ref{lemmmf.1}, we conclude  that $M$ is closed  $R$-invariant.    Consider the operators $\overline{T},\overline{R}\in L(X/M)$ induced      by $T$ and $R,$ respectively.  As $R$ is Riesz  then  \cite[Lemma 15]{kordula-muller} implies  that $R_{M}$ and $\overline{R}$ are  Riesz, and   since $T_{M}$ is onto,   it follows  from  \cite[Corollary 4.2]{aznay-ouahab-zariouh7}  that     $T_{M}\in \mbox{comm}(R_{M}).$  Hence    $(T+R)_{M}$ is lower semi-Browder. Since   $\overline{T}$ is upper semi-Browder, from  \cite[Proposition  4.18]{aznay-ouahab-zariouh7} we conclude  that  $\overline{T+R}=\overline{T}+\overline{R}$ is upper semi-Fredholm. We deduce then from  \cite[Theorem 14]{kordula-muller}  that $T+R$ is essentially semi-regular.\\
(ii)  If   $T$ is semi-regular then $M$ is closed    $Q$-invariant   and $T_{M}$ is onto.   Consider the operators  $\overline{T},\overline{Q}\in L(X/M)$   induced      by $T$ and $Q,$ respectively. As $Q$ is quasi-nilpotent then $Q_{M}$ and $\overline{Q}$ are  quasi-nilpotent. Moreover,  by  Lemma \cite[Lemma 1]{muller}    we have $T_{M}$ is onto and $\overline{T}$ is bounded below. As $T \in \mbox{comm}(TQ)\cap\mbox{comm}(QT)$    then     $\overline{Q}\in\mbox{comm}(\overline{T})$ and  $Q_{M}\in\mbox{comm}(T_{M}).$  Hence    $(T+Q)_{M}$ is onto  and $\overline{T}+\overline{Q}$ is bounded below. Again by  \cite[Lemma 1]{muller}, we deduce that $T+Q$ is semi-regular. 
\end{proof}

\section{  Perturbations by finite rank operators}
We begain this part by the  next lemma which   gives an extension of \cite[Lemma 2.1]{kaashoek-lay}  proved in the case of    commuting operators.
\begin{lem}\label{lem2f.1}  Let $S,T\in L(X)$ such that $S \in \mbox{comm}_{w}(T).$     Then for every integers $m\geq 1$ and $n \geq 3,$ we have   \\
(i)   $\mbox{max}\left\{\displaystyle\mbox{dim}\frac{\mathcal{N}(T^{n})}{\mathcal{N}[(T+S)^{n+m-1}]\cap\mathcal{N}(T^{n})},\displaystyle\mbox{dim}\frac{\mathcal{N}[(T+S)^{n}]}{\mathcal{N}(T^{n+m-1})\cap\mathcal{N}[(T+S)^{n}]}\right\}
\leq\mbox{dim}\,\mathcal{R}(S^{m}).$\\
(ii)   $\mbox{max}\left\{\displaystyle\mbox{dim}\frac{\R(T^{n+m-1})}{\R[(T+S)^{n}]\cap\R(T^{n+m-1})},\displaystyle\mbox{dim}\frac{\R[(T+S)^{n+m-1}]}{\R(T^{n})\cap\R[(T+S)^{n+m-1}]}\right\} \leq\mbox{dim}\,\mathcal{R}(S^{m}).$
\end{lem}
\begin{proof} (i) Since  $S \in \mbox{comm}_{w}(T),$  from  \cite[Corollary  3.6]{aznay-ouahab-zariouh7} we have if  every $x\in \mathcal{N}(T^{n}),$ then   $(T+S)^{n+m-1}x=S^{m}Ax,$ where $A=\displaystyle\sum_{i=0}^{n-1}C_{n+m-1}^{i+m}S^{i}T^{n-i-1}.$   Thus $(T+S)^{n+m-1}(\mathcal{N}(T^{n}))\subset \R(S^{m}).$   Let $M$ be a subspace  such that  $\mathcal{N}(T^{n})=\left(\mathcal{N}[(T+S)^{n+m-1}]\cap\mathcal{N}(T^{n})\right)\oplus M.$  As $T^{n}(T+S)^{n+m-1}=(T+S)^{n+m-1}T^{n}$ then $(T+S)^{n+m-1}(\mathcal{N}(T^{n}))\subset \mathcal{N}(T^{n}).$ And since $M\cap\mathcal{N}[(T+S)^{n+m-1}]=\{0\}$ it then follows that  $\mbox{dim}\,M\leq\mbox{dim}\,\R(S^{m}).$
Since  $S \in \mbox{comm}_{w}(T)$ if and only if $(-S)\in \mbox{comm}_{w}(T+S),$ the proof is complete.\\
(ii)   Let $M$ be a subspace  such that   $\R[(T+S)^{n+m-1}]=\left(\R(T^{n})\cap\R[(T+S)^{n+m-1}]\right)\oplus M$ and  let $(e_{i}:=(T+S)^{n+m-1}v_{i})_{i=1,\dots,k}$ be a  linearly independent family  of  $M.$ From \cite[Corollary  3.6]{aznay-ouahab-zariouh7}, we deduce that
\begin{align*}
(T+S)^{n+m-1}&=\sum_{i=0}^{n+m-1}C_{n+m-1}^{i}T^{n+m-1-i}S^{i}\\
&=\sum_{i=0}^{m-1}C_{n+m-1}^{i}T^{n+m-1-i}S^{i}+\sum_{i=m}^{n+m-1}C_{n+m-1}^{i}T^{n+m-1-i}S^{i}\\
&=T^{n}A+S^{m}B,
\end{align*}
  where $A=\displaystyle\sum_{i=0}^{m-1}C_{n+m-1}^{i}T^{m-1-i}S^{i}$ and $B=\displaystyle\sum_{i=m}^{n+m-1}C_{n+m-1}^{i}T^{n+m-1-i}S^{i-m}.$  If $k>\mbox{dim}\,\mathcal{R}(S^{m}),$ then there exist  $\lambda_{1},\dots,\lambda_{k}$ not all zero such that  $\displaystyle\sum_{i=1}^{k}\lambda_{i}S^{m}Bv_{i}=0.$ Hence  $\displaystyle\sum_{i=1}^{k}\lambda_{i}(T+S)^{n+m-1}v_{i}=\displaystyle\sum_{i=1}^{k}\lambda_{i}T^{n}Av_{i}$ and  thus $\displaystyle\sum_{i=1}^{k}\lambda_{i}e_{i}\in  \R(T^{n})\cap M=\{0\}.$ But this  is a contradiction. Therefore $k\leq \mbox{dim}\,\mathcal{R}(S^{m})$ and then $\mbox{dim}\, M\leq \mbox{dim}\,\mathcal{R}(S^{m}).$ The proof is complete.
\end{proof}
Denote by $\mathcal{F}_{0}(X)$ the class of    power   finite rank operators acting on   $X.$  The next theorem extends  \cite[Theorem 2.2]{kaashoek-lay},  \cite[Proposition 3.1]{oudghiri}  and a special case of the direct implication of \cite[Theorem 3.1]{belhadj}.
\begin{thm}\label{thm3f.1} Let $T \in L(X)$ and $F \in \mathcal{F}_{0}(X)$ such that $F\in\mbox{comm}_{w}(T).$   The following equivalences hold:\\
(i) $q(T)<\infty$ if and only if $q(T+F)<\infty.$ \\
(ii) $p(T)<\infty$ if and only if $p(T+F)<\infty.$\\
(iii) $q_{e}(T)<\infty$ if and only if $q_{e}(T+F)<\infty.$\\
(iv) $p_{e}(T)<\infty$ if and only if $p_{e}(T+F)<\infty.$\\
(v) $m_{T}<\infty$ if and only if $m_{T+F}<\infty.$
\end{thm}

\begin{proof} Let $m\geq 1$ be an integer such that $\mbox{dim}\,\mathcal{R}(F^{m})<\infty.$\\
(i) Assume that $q:=q(T)<\infty$ and let $n\geq \mbox{max}\{3,q\}.$ Then
$$c_{n}:=\displaystyle\mbox{dim}\frac{\R(T^{n+m-1})}{\R[(T+F)^{n}]\cap\R(T^{n+m-1})}=\displaystyle\mbox{dim}\frac{\R(T^{q})}{\R[(T+F)^{n}]\cap\R(T^{q})},$$
 $$c^{'}_{n}:=\displaystyle\mbox{dim}\frac{\R[(T+F)^{n+m-1}]}{\R(T^{n})\cap\R[(T+F)^{n+m-1}]}=\displaystyle\mbox{dim}\frac{\R[(T+F)^{n+m-1}]}{\R(T^{q})\cap\R[(T+F)^{n+m-1}]}.$$
 From  Lemma \ref{lem2f.1}  we have $\mbox{max}\{c_{n},c^{'}_{n}\}<\infty$ for all $n\geq  \mbox{max}\{3,q\}.$  As  $(c_{n})_{n}$ is increasing then    there exists an integer  $k\geq \mbox{max}\{3,q\}$ such that    $\R[(T+F)^{n}]\cap\R(T^{q})=\R[(T+F)^{k}]\cap\R(T^{q}),$ for every $n \geq k.$ Thus $c^{'}_{n}=\displaystyle\mbox{dim}\frac{\R[(T+F)^{n+m-1}]}{\R(T^{q})\cap\R[(T+F)^{k}]}$ for every $n\geq k.$ Therefore $(c^{'}_{n})_{n\geq k}$ is a decreasing sequence. So there exists $r\geq k$ such that for every $n\geq r$ we have  $c^{'}_{n}=c^{'}_{r}.$   Hence $q(T+F)\leq r+m-1$ and  the converse is obvious.   The point (ii) goes similarly.\\
(iii)  Assume that $e:=q_{e}(T)<\infty.$  By Lemma \ref{lem2f.1} we obtain   $\displaystyle\mbox{dim}\frac{\R(T^{e})}{\R[(T+F)^{n}]\cap\R(T^{n+m-1})}=\mbox{dim}\frac{\R(T^{e})}{\R(T^{n+m-1})}+\displaystyle\mbox{dim}\frac{\R(T^{n+m-1})}{\R[(T+F)^{n}]\cap\R(T^{n+m-1})}<\infty$ for every $n\geq l=\mbox{max}\,\{3,e\}.$ Thus $\displaystyle\mbox{dim}\frac{\R(T^{e})}{\R[(T+F)^{n}]\cap\R(T^{e})}<\infty$ for every $n\geq l.$ On the other hand, from the  proof of     Lemma \ref{lem2f.1},  we have $\R[(T+F)^{n+m-1}]\subset \R(T^{e})+\R(F^{m})$ for every $n\geq l.$  Hence    $$\displaystyle\mbox{dim}\frac{\R(T^{e})+\R(F^{m})}{\R[(T+F)^{n+m-1}]}=\displaystyle\mbox{dim}\frac{\R(T^{e})+\R(F^{m})}{\R[(T+F)^{n+m-1}]\cap\R(T^{e})}-\displaystyle\mbox{dim}\frac{\R[(T+F)^{n+m-1}]}{\R[(T+F)^{n+m-1}]\cap\R(T^{e})}<\infty,$$ and consequently  $\displaystyle\mbox{dim}\frac{\R[(T+F)^{n+m-1}]}{\R[(T+F)^{n+m}]}< \infty$ for every $n\geq l.$ Therefore  $q_{e}(T+F)\leq \mbox{max}\,\{m+2,q_{e}(T)+m-1\}<\infty$ and  the converse is obvious. The point (iv) goes similarly. For the proof of the  point (v), the reader is referred to  \cite{aznay-ouahab-zariouh5} in which we mentioned     that $m_{T}=\mbox{min}\,\{p_{e}(T),q_{e}(T)\}.$ 
\end{proof}

The following theorem extends \cite[Lemma 2.1, Lemma 2.2]{zeng}.
\begin{thm}\label{cor2f.1} Let $T \in L(X)$ and  $F \in \mathcal{F}_{0}(X)$ such that $F\in\mbox{comm}_{w}(T).$ Then $T$ is upper semi-B-Fredholm (resp., lower semi-B-Fredholm, B-Fredholm, left Drazin invertible, right Drazin invertible, Drazin invertible)   if and only if $T+F$ is.
\end{thm}
\begin{proof}    Suppose that $T$ is upper semi-B-Fredholm.   Theorem  \ref{thm3f.1} implies that  $\R(T^{p_{e}(T)+1})$ is closed and    $p_{e}(T+F)<\infty.$  Since for every $n\geq 2,$   $FT^{n}=T^{n}F$ then  $F(\mathcal{N}(T^{n}))\subset \mathcal{N}(T^{n}).$  Consider $\tilde{T}$ and $\tilde{F}$ the operators    induced by $T$ and $F$ on $\tilde{X}=X/\mathcal{N}(T^{d+3}),$  where $d=\mbox{dis}(T).$ It is easily seen that  $\tilde{T}$ is   upper semi-Fredholm and  $\tilde{F} \in \mathcal{F}_{0}(\tilde{X}).$ From \cite[Propositon  4.18]{aznay-ouahab-zariouh7}, we deduce that $\tilde{T}+\tilde{F}$ is upper semi-Fredholm.  Hence $\R[(T+F)^{l}]+\mathcal{N}(T^{d+3})$ is closed for every  $l\in \N.$ Furthermore, Lemma \ref{lem2f.1} implies that  $\displaystyle\mbox{dim}\frac{\mathcal{N}(T^{n})}{\mathcal{N}[(T+F)^{n+m-1}]\cap\mathcal{N}(T^{n})}<\infty,$ where $n\geq 3$ and $m\geq 1$    such that $\mbox{dim}\,\mathcal{R}(F^{m})<\infty.$ Hence  $$\displaystyle\mbox{dim}\frac{\mathcal{R}[(T+F)^{n+m-1}]\cap\mathcal{N}(T^{n})}{\mathcal{R}[(T+F)^{n+m-1}]\cap\mathcal{N}[(T+F)^{n+m-1}]\cap\mathcal{N}(T^{n})}<\infty.$$ As $\alpha((T+F)_{[n+m-1]}^{n+m-1})<\infty$  then   $\mbox{dim}\left(\R[(T+F)^{n+m-1}]\cap\mathcal{N}(T^{n})\right)<\infty$ for every integer $n\geq \mbox{max}\{3,p_{e}(T+F)\}.$   From the Neubauer Lemma \cite[Proposition 2.1.1]{labrousse}, we conclude that $\R[(T+F)^{n+m-1}]$ is closed.   Hence $T+F$ is upper semi-B-Fredholm. If $T$ is lower semi-B-Fredholm, then from   \cite{kordula-muller.axiom},  $T^{*}$ is upper semi-B-Fredholm, and consequently $T^{*}+F^{*}$ is upper semi-B-Fredholm. Thus   $T+F$ is lower semi-B-Fredholm (see again \cite{kordula-muller.axiom}).      If $T$ is left Drazin invertible, then $T$ is upper semi-B-Fredholm and $p(T)<\infty.$ So   $T+F$ is upper semi-B-Fredholm, and from Theorem  \ref{thm3f.1} we have $p(T+F)<\infty.$ Thus $T+F$ is left Drazin invertible.  The other cases  go similarly.  Since $F\in\mbox{comm}_{w}(T)$ if and only if $(-F)\in\mbox{comm}_{w}(T+F),$ the proof is complete.
\end{proof}
\begin{cor}\label{corperdrazin} If $T\in L(X)$ is generalized Drazin-Riesz invertible and $F \in \mathcal{F}_{0}(X)$ such that $F\in\mbox{comm}_{l}(T)$ and $T \in \mbox{comm}(TF)$ {\rm[}or $F\in\mbox{comm}_{r}(T)$ and $T \in \mbox{comm}(FT)${\rm]}, then $\sigma_{*}(T)\setminus\{0\}=\sigma_{*}(T+F)\setminus\{0\},$ where  $\sigma_{*}\in\{\sigma_{bf},\sigma_{ubf},\sigma_{lbf},\sigma_{d},\sigma_{ld},\sigma_{rd},\sigma_{gd},\sigma_{g_{z}d}\}.$  If in addition $F\in\mbox{comm}_{w}(T),$ then  $\sigma_{*}(T)=\sigma_{*}(T+F).$
\end{cor}
\begin{proof}  We will leave these routine arguments as exercise for the reader.
\end{proof}
\begin{rema}\label{remanonalg} In \cite[Proposition 3.3]{burgos-oudghiri}, the authors proved that if $X$ is an infinite dimensional complex Banach space and $T\in L(X),$ then there exists   a non-algebraic operator $S\in \mbox{comm}(T).$    From the proof of this result and the one of  \cite[Lemma 3.83]{Aiena1}, it is easy to see  that if in addition  $T$ is an algebraic operator,   then we can consider  $S$ as a  compact operator.
\end{rema}
The next proposition gives a new characterization of power finite rank operators.
\begin{prop}\label{charactfinite} Let   $F \in L(X)$ and   $\sigma_{*}\in\{\sigma_{bf},\sigma_{ubf},\sigma_{lbf},\sigma_{d},\sigma_{ld},\sigma_{rd}\}.$ The following statements are equivalent:\\
(i) $F \in \mathcal{F}_{0}(X);$\\
(ii) $\sigma_{*}(T)=\sigma_{*}(T+F)$ for every generalized Drazin-Riesz invertible operator $T\in\mbox{comm}_{w}(F);$\\
(iii) $\sigma_{*}(T)=\sigma_{*}(T+F)$ for every  $T\in  \mbox{comm}(F).$
\end{prop}
\begin{proof} (i) $\Longrightarrow$ (ii) Is a consequence of  Corollary \ref{corperdrazin}.\\
  (ii) $\Longrightarrow$ (i) We have $\sigma_{*}(F)=\emptyset.$ So   $F$ is algebraic and  $\sigma(F)=\{\lambda_{1},\dots,\lambda_{n}\}.$ Thus $X=X_{1}\oplus\dots \oplus X_{n},$ where $X_{i}=\mathcal{N}((F-\lambda_{i} I)^{m_{i}})$ for some $m_{i}.$ If $F \notin \mathcal{F}_{0}(X),$ then  there exists $1\leq i \leq n$ such that  $\lambda_{i}\neq 0$ and $\mbox{dim}\,X_{i}=\infty.$  As $F_{i}-\lambda_{i}I$  is nilpotent,    from Remark \ref{remanonalg},  there exists a non-algebraic compact  operator  $S_{i}\in \mbox{comm}(F_{i}),$ where $F_{i}=F_{X_{i}}.$  The operator $S=0_{1}\oplus\dots\oplus S_{i} \oplus \dots \oplus 0_{n},$ where $0_{j}=0_{X_{j}},$ is  non-algebraic, compact    and  commutes with $F.$ By hypothesis we have    $\sigma_{*}(S)=\sigma_{*}(S+F),$ this entails, from  \cite[Corollary 2.10]{zeng} that  $\sigma_{*}(S)=\sigma_{*}(S_{i})=\sigma_{*}(S_{i}+F_{i})=\sigma_{*}(S_{i}+\lambda_{i}I)=\sigma_{*}(S+F),$ since $F_{i} -\lambda_{i} I$ is nilpotent.  Hence $\lambda_{i}=0$ and this  is a contradiction.  Thus  $F\in\mathcal{F}_{0}(X).$ \\ 
 (i) $\Longrightarrow$ (iii) Is a consequence of Theorem  \ref{cor2f.1}, and   (iii) $\Longrightarrow$ (i)  is proved in  \cite[Theorem 2.11]{zeng}.
\end{proof}
The next proposition  extends the second assertion of  \cite[Corollary 3.5]{berkani}.  $\mathcal{F}(X)$ denotes the ideal of finite rank operator in $L(X).$
\begin{prop} Let $S,T\in L(X)$ be  B-Fredholm operators, if        $S\in\mbox{comm}_{r}(T)\cup\mbox{comm}_{l}(T),$   then $TS$ is B-Fredholm.
\end{prop}
\begin{proof}  Assume that $S\in\mbox{comm}_{r}(T),$ then $S^{*}\in\mbox{comm}_{l}(T^{*}).$ As $T$ and $S$ are  B-Fredholm then  $T^{*}$ and $S^{*}$ are  B-Fredholm. From \cite[Theorem 3.4]{berkani}, $\pi_{f}(S^{*}):=S^{*}+\mathcal{F}(X)$ and $\pi_{f}(T^{*}):=T^{*}+\mathcal{F}(X)$ are Drazin invertible.  Since   $\pi_{f}(S^{*})\in\mbox{comm}_{l}(\pi_{f}(T^{*})),$  from Proposition \ref{productDrazin} we get  $\pi_{f}(S^{*}T^{*}):=S^{*}T^{*}+\mathcal{F}(X)$ is Drazin invertible.  We conclude again by  \cite[Theorem 3.4]{berkani}   that $TS$ is a B-Fredholm operator. The case $S\in\mbox{comm}_{l}(T)$ is analogous.
\end{proof}

As a continuation to what has been done in the paper \cite{aznay-ouahab-zariouh6}, we end this part by the following theorem which  improves \cite[Theorem 2.1]{aiena-garcia}.

\begin{thm} Let $T \in L(X).$ Then  $$\sigma_{d}(T)=\sigma_{bf}(T)\cup\mbox{acc}(\mbox{acc}\,\sigma(T)),$$  $$\sigma_{b}(T)=\sigma_{e}(T)\cup\mbox{acc}(\mbox{acc}\,\sigma(T)).$$ 
\end{thm}
\begin{proof} Let us prove that  $\sigma_{d}(T)=\sigma_{bf}(T)\cup\mbox{acc}(\mbox{acc}\,\sigma(T)).$  Let $\lambda \notin \sigma_{bf}(T)\cup\mbox{acc}(\mbox{acc}\,\sigma(T))$ and  without loss of generality we can assume that $\lambda =0.$ Then $T$ is B-Fredholm and $0\notin \mbox{acc}\,(\mbox{acc}\,\sigma(T)).$ This entails from \cite[Theorem 2.21]{aznay-ouahab-zariouh4} and  \cite[Theorem 4.11]{aznay-ouahab-zariouh6}   that $T$ is a $g_{z}$-invertible operator and $T=T_{M}\oplus T_{N}$ for some $(M,N)\in \mbox{Red}(T)$ such that $T_{M}$ is semi-regular and $T_{N}$ is nilpotent. This  implies  again by \cite[Theorem 4.7]{aznay-ouahab-zariouh6} that $p(T_{M})=q(T_{M})=\tilde{p}(T)=\tilde{q}(T)=0,$ and so  $T_{M}$ is  invertible. Hence $T$ is Drazin invertible. The converse is clear, since $\mbox{acc}\,\sigma(T)\subset \sigma_{d}(T).$  The second equality goes similarly. For the definition of $ \tilde{p}(T)$ and  $\tilde{q}(T)$ of a  $g_{z}$-invertible operator $T,$ see \cite{aznay-ouahab-zariouh6}.
\end{proof}

\goodbreak 
{\small \noindent Zakariae Aznay,\\  Laboratory (L.A.N.O), Department of Mathematics,\\Faculty of Science, Mohammed I University,\\  Oujda 60000 Morocco.\\
aznay.zakariae@ump.ac.ma\\

\noindent Abdelmalek Ouahab,\newline Laboratory (L.A.N.O), Department of
	Mathematics,\newline Faculty of Science, Mohammed I University,\\
	\noindent Oujda 60000 Morocco.\\
	\noindent ouahab05@yahoo.fr\\

 \noindent Hassan  Zariouh,\newline Department of
Mathematics (CRMEFO),\newline
 \noindent and laboratory (L.A.N.O), Faculty of Science,\newline
  Mohammed I University, Oujda 60000 Morocco.\\
 \noindent h.zariouh@yahoo.fr

\end{document}